\documentclass[11pt, reqno]{amsart}

\usepackage{fullpage}
\usepackage{hyperref}
\usepackage{appendix}

\numberwithin{equation}{section}

{\theoremstyle{plain}
\newtheorem{theorem}{Theorem}[section]
\newtheorem{proposition}[theorem]{Proposition}
\newtheorem{corollary}[theorem]{Corollary}
\newtheorem{lemma}[theorem]{Lemma}}

{\theoremstyle{remark}   
\newtheorem{definition}[theorem]{Definition}
\newtheorem{remark}[theorem]{Remark}
\newtheorem{example}[theorem]{Example}}

\newcommand{\C}{\mathbb{C}}

\newcommand\blfootnote[1]{%
\begingroup
\renewcommand\thefootnote{}\footnote{#1}%
\addtocounter{footnote}{-1}%
\endgroup
}

\author{Fran\c{c}ois Legrand}
\email{francois.legrand@unicaen.fr}
\address{Normandie Univ., UNICAEN, CNRS, Laboratoire de Math\'ematiques Nicolas Oresme, 14000 Caen, France}

\author{Elad Paran}
\email{paran@openu.ac.il}
\address{Department of Mathematics and Computer Science, the Open University of Israel, Ra'anana 4353701, Israel}

\title{L\"uroth's and Igusa's theorems over Division Rings}

\begin{document}

\maketitle

\vspace{-1cm}

\begin{abstract}
Let $H$ be a division ring of finite dimension over its center, let $H[T]$ be the ring of polynomials in a central variable over $H$, and let $H(T)$ be its quotient skew field. We show that every intermediate division ring between $H$ and $H(T)$ is itself of the form $H(f)$, for some $f$ in the center of $H(T)$. This generalizes the classical L\"uroth's theorem. More generally, we extend Igusa's theorem characterizing the transcendence degree 1 subfields of rational function fields, from fields to division rings. 
\end{abstract}

\section{Introduction} \label{sec:intro}

\blfootnote{2020 Mathematics Subject Classification. Primary 12E15 12F20 16K20.
}Let $K$ be an arbitrary field and let $K(X)$ be the field of rational functions over $K$. L\"uroth's theorem states that every intermediate field $K \subset F \subseteq K(X)$ is itself a rational function field over $K$. The theorem was first proven for $K = \C$ by L\"uroth in \cite{Lue75}, and for a general field $K$ by Steinitz in \cite{Ste10}. This result is foundational for general field theory and for the theory of algebraic curves, see \cite[\S1.3]{Sha13}. The theorem was generalized to transcendence degree 1 subfields of rational function fields in any number of variables by Gordan in \cite{Gor87} for fields of characteristic $0$, and in general by Igusa in \cite{Igu51}. Over the years, various proofs, employing different approaches, have been given to these results, e.g., in \cite[p. 106]{Che51}, \cite{Sam53}, \cite{Iit79}, \cite{OS22}.

In the present work, we study L\"uroth's and Igusa's theorems in the more general context of division rings.

Let $H$ be a division ring (a.k.a. a skew field, or a division algebra over its center field), and let $H[T]$ be the ring of polynomials over $H$ in a central variable $T$. The ring $H[T]$ is an Ore domain, hence admits a unique quotient skew field $H(T)$. The arithmetic of the ring $H[T]$ (and more generally, of skew polynomial rings) is well-studied (see the classical works \cite{Ore33}, \cite{GM65}, and the modern works of Lam and Leroy \cite{Lam86}, \cite{LL88}, \cite{LL94}, \cite{LL04}, \cite{LLO08}, for example), however the study of the quotient skew field $H(T)$ is not as expansive. Our main result is the following one, whose first part generalizes L\"uroth's theorem, and its second part generalizes Igusa's theorem.

\begin{theorem}\label{thm:main_1} 
Let $H$ be a division ring of finite dimension over its center $Z(H)$, let $n$ be a positive integer and let $H \subseteq L \subseteq H(T_1, \dots, T_n)$ be an intermediate division ring.

\vspace{0.5mm}

\noindent
{\rm{(1)}} Assume $n=1$ and $L \not=H$. Then there exists $f \in Z(H)(T_1) \setminus Z(H)$ such that $L = H(f)$.

\vspace{0.5mm}

\noindent
{\rm{(2)}} Assume there is $g \in Z(H)(T_1, \dots, T_n)$ such that $L/H(g)$ is algebraic and such that $g$ is not algebraic over $H$. Then there exists $f \in Z(H)(T_1, \dots, T_n) \setminus Z(H)$ such that $L = H(f)$.
\end{theorem}

Here $T_1,\ldots,T_n$ denote independent central variables and $H(f)$ denotes the division ring genera\-ted by $f$ over $H$ inside $H(T_1,\ldots,T_n)$ -- the terminology and notations are reviewed in detail in \S\ref{sec:prelim}. 

In order to prove the theorem, we first prove several claims concerning extensions of division rings, that eventually allow us to deduce the theorem from the classical, commutative version of it. This is done in \S\ref{sec:main_1}. 

After concluding the proof of Theorem \ref{thm:main_1}, we turn out attention to the more general skew polynomial ring $H[T,\sigma]$, where $\sigma$ is an automorphism of the division ring $H$. In the present context, it is natural to ask whether a version of L\"uroth's theorem holds for the quotient skew field $H(T,\sigma)$ of $H[T,\sigma]$. In the case where $\sigma$ is an inner automorphism, we observe that such a version immediately follows from our main result, see Remark \ref{rk:inner}. However, the case where $\sigma$ is an arbitrary automorphism seems more difficult, even in the simplest case where $H$ itself is a field and $\sigma$ is of order $2$. In \S\ref{sec:main_2}, we consider the skew polynomial ring $\C[T,\sigma]$, where $\sigma$ is the complex conjugation, and its quotient skew field $\C(T,\sigma)$, and study the intermediate division rings between $\C$ and $\C(T,\sigma)$. We prove that any intermediate division ring $\C \subseteq D \subseteq \C(T,\sigma)$ is of the form $\C(f)$, provided that $D$ is $\sigma$-invariant, see Definition \ref{def:invariant} and Theorem \ref{thm:main_2}. However, we observe that $\sigma$-invariance is only a sufficient condition for $D$ to be of this form, not a necessary one, see Proposition \ref{converse}. It remains an open question whether L\"uroth's theorem holds unconditionally for $\C(T,\sigma)$, and more generally for skew fields of the form $H(T,\sigma)$, see also Example \ref{referee_example}. Another open question is whether Theorem \ref{thm:main_1} holds for division rings of infinite dimension over their center.

\vspace{2mm}

{\bf{Acknowledgments.}} We thank the anonymous referee for his/her comments, and in particular for providing us with Example \ref{referee_example}. We thank Adam Chapman for his help with Lemma \ref{lemma:finite}. This work fits into Project TIGANOCO ({\it{Th\'eorie Inverse de GAlois NOn COmmutative}}), which is funded by the European Union within the framework of the Operational Programme ERDF/ESF 2014-2020. 

\section{Preliminaries} \label{sec:prelim}

In this section, we collect the basic material on division rings which will be used in the sequel.

A {\it{division ring}} is a (unital, associative) ring $H$ in which every non-zero element is invertible. Given a non-zero element $b$ of a division ring $H$, the conjugation by $b$ in $H$ is an {\it{inner}} automorphism of $H$, denoted by $I_H(b)$, and we let ${\rm{Inn}}(H)$ denote the group of all inner automorphisms of $H$.

For the next three items, we fix an extension $L/H$ of division rings (i.e., $H \subseteq L$).

\noindent
$\bullet$ The {\it{degree}} $[L:H]$ of $L/H$ is the dimension of the (left) $H$-linear space $L$ and $L/H$ is {\it{finite}} if $[L : H] < \infty$. We also say that $H$ is {\it{centrally finite}} if $H$ is a finite extension of its center $Z(H)$.

\noindent
$\bullet$ Given a subset $S$ of $L$, we let $H(S)$ denote the intersection of all division rings contained in $L$ and which contain both $H$ and $S$, and we simply write $H(x)$ if $S=\{x\}$. We then say that an element $x$ of $L$ is {\it{algebraic}} over $H$ if $H(x)/H$ is finite and that $L/H$ is {\it{algebraic}} if every element of $L$ is algebraic over $H$.

\noindent
$\bullet$ Letting ${\rm{Aut}}(L/H)$ denote the group of all automorphisms of $L$ fixing $H$ point-wise, we say that $L/H$ is {\it{outer}} if ${\rm{Inn}}(L) \cap {\rm{Aut}}(L/H) = \{{\rm{id}}_L\}$. Equivalently, if $C_L(H)$ denotes the {\it{centralizer}} of $H$ in $L$, i.e., $C_L(H) = \{x \in L : xy=yx \, \, (y \in H)\}$, then $L/H$ is outer if and only if $C_L(H)$ equals the center $Z(L)$ of $L$. In particular, if $L/H$ is outer, then $Z(H) \subseteq Z(L)$.

\begin{lemma} \label{lemma:outer}
Let $L/H$ be an outer extension of division rings.

\vspace{0.5mm}

\noindent
{\rm{(1)}} The extensions $L/F$ and $F/H$ are outer for every intermediate division ring $H \subseteq F \subseteq L$.

\vspace{0.5mm}

\noindent
{\rm{(2)}} Assume $H$ is centrally finite and $L/H$ is algebraic. Then $Z(L)/Z(H)$ is algebraic.
\end{lemma}

\begin{proof}
(1) First, if $I_L(a)$ ($a \in L \setminus \{0\}$) is in ${\rm{Aut}}(L/F)$, then $I_L(a) \in {\rm{Aut}}(L/H)$. As $L/H$ is outer, we get $I_L(a) = {\rm{id}}_L$. Next, fix $a \in F \setminus \{0\}$ such that $I_F(a) \in {\rm{Aut}}(F/H)$. Then $I_L(a) \in {\rm{Aut}}(L/H)$ and, as $L/H$ is outer, we get that $I_L(a) = {\rm{id}}_L$. In particular, $I_F(a) = {\rm{id}}_F$.

\vspace{0.5mm}

\noindent
(2) Fix $x \in Z(L)$. As $L/H$ is algebraic, $x$ is in some intermediate division ring $H \subseteq F \subseteq L$ with $F/H$ finite. Then, as $H$ is centrally finite, $F/Z(H)$ is finite. Thus $F \cap Z(L)$ is a division ring which contains $Z(H)$ and $x$, which is contained in $Z(L)$ and which is a finite extension of $Z(H)$.
\end{proof}

\begin{lemma} \label{lemma:finite}
Let $H \subseteq L$ be division rings with $L$ centrally finite. Then $H$ is centrally finite.
\end{lemma}

\begin{proof}[Comments on proof]
The statement is well-known to experts but the authors could not find an explicit reference for it in the literature. For the sake of completeness, we provide a proof of the lemma in \S\ref{sec:finite}.
\end{proof}

A non-zero ring $R$ with no zero divisor is a {\it{right Ore domain}} if, for any $x, y \in R \setminus \{0\}$, there exist $r, s \in R \setminus \{0\}$ with $xr = ys$. If $R$ is a right Ore domain, there is a division ring $H$ which contains $R$ and every element of which can be written as $ab^{-1}$ with $a \in R$ and $b \in R \setminus \{0\}$ (see, e.g., \cite[Theo\-rem 6.8]{GW04}). Moreover, $H$ is unique up to isomorphism (see, e.g., \cite[Proposition 1.3.4]{Coh95}).

Given a ring $H$ and an automorphism $\sigma$ of $H$, the polynomial ring $H[T, \sigma]$ is the ring of all polynomials of the form $a_0 + a_1 T + \cdots + a_n T^n$ with $n \geq 0$ and $a_0, \dots, a_n \in H$, whose addition is defined component-wise and whose multiplication is given by the usual rule
$$\bigg(\sum_{i=0}^n a_i T^i \bigg) \cdot \bigg(\sum_{j=0}^m b_j T^j \bigg) = \sum_{k=0}^{n+m} \bigg(\sum_{\ell=0}^k a_\ell \sigma^\ell (b_{k-\ell}) \bigg) T^k.$$
In the sense of Ore (see \cite{Ore33}), the ring $H[T, \sigma]$ is the polynomial ring $H[T, \sigma, \delta]$ in the variable $T$, where the $\sigma$-derivation $\delta$ is the zero derivation.

For the rest of this section, we fix a division ring $H$. Then $H[T, \sigma]$ has no zero divisors, as the degree is additive on products. Moreover, $H[T, \sigma]$ is a right Ore domain (see, e.g., \cite[Theorem 2.6 and Corollary 6.7]{GW04}). The unique division ring which contains $H[T, \sigma]$ and each element of which can be written as $ab^{-1}$ with $a \in H[T, \sigma]$ and $b \in H[T, \sigma] \setminus \{0\}$ is then denoted by $H(T, \sigma)$.

If $\sigma={\rm{id}}_H$, we write $H[T]$ and $H(T)$ for simplicity, and the variable $T$ is a central element of $H(T)$. One can iteratively construct polynomial rings in several central variables over $H$, by putting $H[T_1,T_2] = H[T_1][T_2]$, $H[T_1,T_2,T_3] = H[T_1,T_2][T_3]$, and so on. As the variables are central, the order in which they are added does not change the ring obtained. Furthermore, by an easy induction, the ring $H[T_1, \dots, T_n]$ in $n$ central variables over $H$ is a right Ore domain for every $n \geq 1$. The unique division ring which contains $H[T_1, \dots, T_n]$ and each element of which can be written as $ab^{-1}$ with $a \in H[T_1, \dots, T_n]$ and $b \in H[T_1, \dots, T_n] \setminus \{0\}$ is then denoted by $H(T_1, \dots, T_n)$. In the sequel, we will constantly use the equalities $H(T_1, \dots, T_n, T_{n+1}) = H(T_1, \dots, T_n)(T_{n+1})$ and $Z(H(T_1, \dots, T_n)) = Z(H)(T_1, \dots, T_n)$, and that $H(T_1, \dots, T_n)$ is centrally finite if $H$ is ($n \geq 1$). See, e.g., \cite[Propositions 2 and 3]{ALP20} for more details.

We also consider the division ring $H((T))$ of Laurent series of the form $\sum_{i \geq i_0} a_i T^i$, where $i_0 \in \mathbb{Z}$ and $a_i \in H$ for $i \geq i_0$, whose addition and multiplication are defined component-wise and by 
$$\bigg(\sum_{i \geq i_0} a_i T^i \bigg) \bigg( \sum_{i \geq i_1} b_i T^i \bigg) = \sum_{i \geq i_0 + i_1} \bigg(\sum_{\ell=0}^{i-i_0-i_1} a_{i_0 + \ell} \, b_{i-i_0 - \ell} \bigg) T^i,$$
respectively. Since $H[T]$ is a right Ore domain contained in $H((T))$, we have $H(T) \subseteq H((T))$. See, e.g., \cite[\S2.3]{Coh95} for more details on division rings of Laurent series.

\section{Proof of Theorem \ref{thm:main_1}} \label{sec:main_1}

To prove Theorem \ref{thm:main_1}, we will require the next four lemmas. The first one characterizes finite outer extensions of centrally finite division rings.

\begin{lemma} \label{brubru}
Let $H$ be a centrally finite division ring and $L$ a finite extension of $H$ with $Z(H) \subseteq Z(L)$. The next three conditions are equivalent:

\vspace{0.5mm}

\noindent
{\rm{(i)}} the unique $Z(H)$-linear map $\psi : H \otimes_{Z(H)} Z(L) \rightarrow L$ which fulfills $\psi(x \otimes y) = xy$ for every $(x,y) \in H \times Z(L)$ is an isomorphism of $Z(H)$-algebras,

\vspace{0.5mm}

\noindent
{\rm{(ii)}} $L/H$ is outer,

\vspace{0.5mm}

\noindent
{\rm{(iii)}} $[Z(L) : Z(H)]=[L:H]$.
\end{lemma}

\begin{proof}
First, as the map
$$\left \{ \begin{array} {ccc}
H \times Z(L) & \rightarrow & L \\
(x,y) & \mapsto & xy
\end{array} \right. $$
is $Z(H)$-bilinear, $\psi$ is well-defined, and it is a morphism of $Z(H)$-algebras. Moreover, $\psi$ is injective (see, e.g., \cite[Proposition 2.36]{Kna07}). Furthermore, as
$${\rm{dim}}_H ({\rm{Im}}(\psi)) \times [H : Z(H)] = {\rm{dim}}_{Z(L)} ({\rm{Im}}(\psi)) \times [Z(L) : Z(H)]$$ 
and as the left-hand side is finite, $[Z(L) : Z(H)]$ is also finite and hence ${\rm{dim}}_H ({\rm{Im}}(\psi)) = [Z(L) : Z(H)].$ Then (i) $\Leftrightarrow$ (iii) is clear. Next, the centralizer of ${\rm{Im}}(\psi)$ in $L$ is the centralizer $C_L(H)$ of $H$ in $L$. As $L$ is centrally finite, the {\it{double centralizer theorem}} (see, e.g., \cite[Theorem 2.43]{Kna07}) then yields $[L : Z(L)] = [{C}_L(H) :Z(L)] \times {\rm{dim}}_{Z(L)} ({\rm{Im}}(\psi)).$ Hence, (i) $\Leftrightarrow$ (ii), as needed.
\end{proof}

Our second lemma is a variant of the first one.

\begin{lemma} \label{lemma:algebraic}
Let $L/H$ be an outer extension of division rings such that $L$ is centrally finite. Then the unique $Z(H)$-linear map $\psi : H \otimes_{Z(H)} Z(L) \rightarrow L$ which fulfills $\psi(x \otimes y) = xy$ for every $(x,y) \in H \times Z(L)$ is an isomorphism of $Z(H)$-algebras.
\end{lemma}

\begin{proof}
First, $Z(H) \subseteq Z(L)$ as $L/H$ is outer. Then, as in the proof of Lemma \ref{brubru}, the map $\psi$ is a well-defined monomorphism of $Z(H)$-algebras. Moreover, as $L$ is centrally finite, $H$ is centrally finite, by Lemma \ref{lemma:finite}. Furthermore, ${\rm{Im}}(\psi)$ is a ring with no zero divisors and a $Z(L)$-linear space with finite dimension $[H:Z(H)]$. Hence, ${\rm{Im}}(\psi)$ is a division ring (see, e.g., \cite[Proposition 3.1.2]{Coh95}), whose center equals $Z(L)$. Using that $L/Z(L)$ is finite, we get that $L/{\rm{Im}}(\psi)$ is finite and that ${\rm{Im}}(\psi)$ is centrally finite. Moreover, since $L/H$ is outer and since ${\rm{Im}}(\psi)$ is an intermediate division ring, $L/{\rm{Im}}(\psi)$ is outer, by Lemma \ref{lemma:outer}(1). Thus, $L={\rm{Im}}(\psi)$ by Lemma \ref{brubru}.
\end{proof}

Our third lemma shows that every extension of the form $H(T_1, \dots, T_n)/H$ is outer.

\begin{lemma} \label{lemma:center}
{\rm{(1)}} Let $L/H$ be an outer extension of division rings. Then $L((T))/H$ is outer.

\vspace{0.5mm}

\noindent
{\rm{(2)}} The extension $H(T_1, \dots, T_n)/H$ is outer for every $n \geq 1$.
\end{lemma}

\begin{proof}
(1) Let $a \in L((T))$ be such that $ac=ca$ for every $c \in H$. Set $a = \sum_{i \geq i_0} a_i T^i$, where $i_0 \in \mathbb{Z}$ and $a_i \in L$ for every $i \geq i_0$. As $ac=ca$ for every $c \in H$, we have $ca_i = a_i c$ for every $c \in H$, i.e., every $a_i$ lies in the centralizer $C_L(H)$ of $H$ in $L$. As the extension $L/H$ is outer, we actually have $C_L(H) = Z(L)$ and hence $a \in Z(L)((T)) \subseteq Z(L((T)))$, as needed.

\vspace{1mm}

\noindent
(2) We proceed by induction on $n \geq 1$. For $n=1$, we get from (1) that $H((T_1))/H$ is outer. As $H(T_1)$ is an intermediate division ring, Lemma \ref{lemma:outer}(1) yields that $H(T_1)/H$ is outer. Now, fix $n \geq 1$ and assume $H(T_1, \dots, T_n)/H$ is outer. We may then use (1) to get that $H(T_1,\dots, T_n) ((T_{n+1}))/H$ is outer. Since $H(T_1, \dots, T_n)(T_{n+1})$ is an intermediate division ring which equals $H(T_1, \dots, T_{n+1})$, Lemma \ref{lemma:outer}(1) yields that $H(T_1, \dots, T_{n+1})/H$ is outer, thus ending the proof.
\end{proof}

Our last lemma describes the division ring generated by a central element of $H(T_1, \dots, T_n)$, if $H$ is centrally finite.

\begin{lemma} \label{lemma:dl}
Let $H$ be a centrally finite division ring, $n \geq 1$ and $f \in Z(H)(T_1, \dots, T_n)$. Then the unique $Z(H)$-linear map $\psi : H \otimes_{Z(H)} Z(H)(f) \rightarrow H(f)$ which fulfills $\psi(x \otimes y) = xy$ for every $(x, y) \in H \times Z(H)(f)$ is an isomorphism of $Z(H)$-algebras. In particular, $Z(H(f)) = Z(H)(f)$.
\end{lemma}

\begin{proof}
First, by Lemma \ref{lemma:outer}(1) and Lemma \ref{lemma:center}(2), the extension $H(f)/H$ is outer. In particular, $Z(H) \subseteq Z(H(f))$. Moreover, as $f \in Z(H)(T_1, \dots, T_n)$, we have $f \in Z(H(f))$ and hence the inclusions $Z(H) \subseteq Z(H)(f) \subseteq Z(H(f))$ hold.

Now, as in the proof of Lemma \ref{brubru}, the map $\psi$ is well-defined and injective. Moreover, as $Z(H)(f) \subseteq Z(H(f))$, it is a morphism of $Z(H)$-algebras. Furthermore, ${\rm{Im}}(\psi)$ is both a ring with no zero divisors and a $Z(H)(f)$-linear space with finite dimension $[H:Z(H)]$. Hence, ${\rm{Im}}(\psi)$ is a division ring, which is contained in $H(T_1, \dots, T_n)$ and which contains $H$ and $f$. Therefore, $H(f) = {\rm{Im}}(\psi)$, thus ending the proof.
\end{proof}

\begin{proof}[Proof of Theorem \ref{thm:main_1}]
We prove each statement separately, by reducing to the commutative case.

\vspace{1.5mm}

\noindent
(1) First, by Lemma \ref{lemma:outer}(1) and \ref{lemma:center}(2), the extensions $H(T_1)/L$ and $L/H$ are outer and, hence, $Z(H) \subseteq Z(L) \subseteq Z(H)(T_1)$. We may then consider the unique $Z(H)$-linear map $\psi : H \otimes_{Z(H)} Z(L) \rightarrow L$ which fulfills $\psi(x \otimes y) = xy$ for every $(x,y) \in H \times Z(L)$. Moreover, since $H(T_1)$ is centrally finite, $L$ is centrally finite by Lemma \ref{lemma:finite}. We may then apply Lemma \ref{lemma:algebraic} to get
\begin{equation} \label{Lpsi}
L = {\rm{Im}}(\psi).
\end{equation}
Next, assume $Z(L) = Z(H)$. Then ${\rm{Im}}(\psi)=H$, i.e., $L=H$ by \eqref{Lpsi}, which cannot hold. Fi\-nal\-ly, by L\"uroth's theorem and as $Z(L) \not=Z(H)$, there is $f \in Z(H)(T_1) \setminus Z(H)$ such that $Z(L) = Z(H)(f)$. Hence, by Lemma \ref{lemma:dl}, we get ${\rm{Im}}(\psi) = H(f)$, i.e., $L=H(f)$ by \eqref{Lpsi}.

\vspace{1.5mm}

\noindent
(2) First, by Lemma \ref{lemma:outer}(1) and \ref{lemma:center}(2), the extensions $H(T_1, \dots, T_n)$ and $L/H$ are outer and, hence, $Z(H) \subseteq Z(L) \subseteq Z(H)(T_1, \dots, T_n)$. We may then consider the unique $Z(H)$-linear map $\psi : H \otimes_{Z(H)} Z(L) \rightarrow L$ which fulfills $\psi(x \otimes y) = xy$ for every $(x,y) \in H \times Z(L)$. Moreover, as $H(T_1, \dots, T_n)$ is centrally finite, $L$ is centrally finite (see Lemma \ref{lemma:finite}). Lemma \ref{lemma:algebraic} then yields
\begin{equation} \label{Lpsi2}
L = {\rm{Im}}(\psi).
\end{equation}
Moreover, $H(g)$ is centrally finite (again by Lemma \ref{lemma:finite}) and, by Lemma \ref{lemma:outer}(2), the extension $Z(L)$ of $Z(H(g)) = Z(H)(g)$ (see Lemma \ref{lemma:dl} for the last equality) is algebraic. Furthermore, if $g$ is algebraic over $Z(H)$, then $Z(H)(g)/Z(H)$ is finite. As $H(g)$ is centrally finite, we get that $H(g)/H$ is finite, which cannot hold. Hence, $Z(L)/Z(H)$ has transcendence degree 1. We may then apply Igusa's generalization of L\"uroth's theorem to get that there is $f \in Z(H)(T_1, \dots, T_n) \setminus Z(H)$ with $Z(L)=Z(H)(f).$ Hence, by Lemma \ref{lemma:dl}, we get ${\rm{Im}}(\psi) = H(f)$, i.e., $L=H(f)$ by \eqref{Lpsi2}.
\end{proof}

\begin{remark} \label{rk:inner}
Given a division ring $H$ and $\sigma \in {\rm{Inn}}(H)$, say $\sigma = I_H(b)$ with $b \in H \setminus \{0\}$, we let $H[b^{-1} T]$ denote the intersection of all subrings of $H[T, \sigma]$ containing $H$ and $b^{-1}T$. Then 
\begin{equation} \label{eq:change}
H[b^{-1} T] = \{a_0 + a_1 (b^{-1} T) + \cdots + a_n (b^{-1} T)^n : n \geq 0, a_0, \dots, a_n \in H\} = H[T, \sigma]
\end{equation}
and, as the $(b^{-1} T)^n$'s ($n \geq 0$) are linearly independent over $H$, the ring $H[b^{-1} T]$ is the polynomial ring in the central variable $b^{-1}T$ over $H$. As $H(T, \sigma) = H(b^{-1}T)$ by \eqref{eq:change}, we get the next extension of Theorem \ref{thm:main_1}:

\vspace{1mm}

\noindent
{\it{Let $H$ be a centrally finite division ring and let $n$ be a positive integer. Fix $\sigma_1 \in {\rm{Inn}}(H)$ and, for every $i \in \{2, \dots, n\}$, fix $\sigma_i \in {\rm{Inn}}(H(T_1, \sigma_1)\cdots (T_{i-1}, \sigma_{i-1}))$. Let $H \subseteq L \subseteq H(T_1, \sigma_1) \cdots (T_n , \sigma_n)$ be an intermediate division ring.

\vspace{0.5mm}

\noindent
{\rm{(1)}} Assume $n=1$ and $L \not=H$. Then there is $f \in Z(H(T_1, \sigma_1)) \setminus Z(H)$ such that $L = H(f)$.

\vspace{0.5mm}

\noindent
{\rm{(2)}} Assume there is $g \in Z(H(T_1, \sigma_1) \cdots (T_n, \sigma_n))$ such that $L/H(g)$ is algebraic and such that $g$ is not algebraic over $H$. Then there is $f \in Z(H(T_1, \sigma_1) \cdots (T_n, \sigma_n)) \setminus Z(H)$ such that $L = H(f)$.}}
\end{remark}

\section{The case of $\mathbb{C}(T, \sigma)$}\label{sec:main_2}

In this section, we consider the division ring $\mathbb{C}(T, \sigma)$, where $\sigma$ denotes the complex conjugation. Let $\mathbb{H}$  denote Hamilton's quaternions algebra, i.e., $\mathbb{H} = \mathbb{R} \oplus \mathbb{R}i \oplus \mathbb{R}j \oplus \mathbb{R} k$ where $i^2 = j^2 = k^2 =ijk=-1$. From now on, we view $\mathbb{C}$ as the subring $\mathbb{R} \oplus \mathbb{R}i$ of $\mathbb{H}$.

\begin{lemma}\label{embedding} 
Consider the map 
$$\phi : \left \{ \begin{array} {ccc}
\mathbb{C}[T,\sigma] & \longrightarrow & \mathbb{H}[X] \\
a_0 + a_1 T + \cdots + a_nT^n & \longmapsto & a_0 + a_1 j X + \cdots + a_n j^n X^n
\end{array} \right.. $$
Then $\phi$ is a ring homomorphism, which is injective and which fixes $\mathbb{C}$ point-wise.
\end{lemma}

\begin{proof}
Clearly, $\phi$ fixes $\mathbb{C}$ point-wise and $\phi$ is both additive and injective. Hence, to conclude the proof, it suffices to check that $\phi$ is multiplicative on monomials. To that end, consider two monomials $aT^n$ and $bT^m$, with $n, m \in \mathbb{N}$ and $a,b \in \mathbb{C}$. By the definition of $\phi$, we have
\begin{equation} \label{eq:presented}
\phi(aT^nbT^m) = \phi(a \sigma^n(b)T^{n+m}) = a\sigma^n(b)j^{n+m}X^{n+m}.
\end{equation}
By considering the four different possible residues of $n$ modulo $4$, one checks that $\sigma^n(b)j^n = j^nb$ for all $n \geq 0$. Thus, \eqref{eq:presented} yields $\phi(aT^nbT^m) = aj^nbj^mX^{n+m} = aj^nX^{n}bj^mX^{m} = \phi(aT^n)\phi(bT^m),$ as needed. 
\end{proof}

By, e.g., \cite[Proposition 6.3]{GW04} and since $\mathbb{C}[T, \sigma]$ is a right Ore domain, the monomorphism $\phi$ from Lemma \ref{embedding} extends to the following monomorphism:
$$\phi : \left \{ \begin{array} {ccc}
\mathbb{C}(T, \sigma) & \rightarrow & \mathbb{H}(X)\\
P Q^{-1} & \mapsto & \phi(P) \phi(Q)^{-1}
\end{array} \right.. $$

Now, we extend $\sigma$ to $\mathbb{H}$ by setting 
$$\sigma(a+bi+cj+dk) = a-bi+cj-dk$$ 
for $a,b,c,d \in \mathbb{R}$. Then $\sigma$ is an automorphism of $\mathbb{H}$ of order 2. Moreover, we may extend $\sigma$ to an automorphism of order 2 of $\mathbb{H}[X]$, by setting
$$\sigma(a_0 + a_1 X + \cdots + a_n X^n) = \sigma(a_0) + \sigma(a_1) X + \cdots +  \sigma(a_n) X^n$$
for $n \geq 0$ and $a_0, \dots, a_n \in \mathbb{H}$. Finally, $\sigma$ extends to the next automorphism of order 2 of $\mathbb{H}(X)$:
$$\sigma : \left \{ \begin{array} {ccc}
\mathbb{H}(X) & \rightarrow & \mathbb{H}(X)\\
P Q^{-1} & \mapsto & \sigma(P) \sigma(Q)^{-1}
\end{array} \right.. $$

\begin{definition}\label{def:invariant}
We say that a subset $S$ of $\mathbb{H}(X)$ is {\it{$\sigma$-invariant}} if $\sigma(S)= S$.
\end{definition}

\begin{example} \label{example}
Given a subset $S$ of $\mathbb{C}(T, \sigma)$, we have $\phi(\mathbb{C}(S)) = \mathbb{C}(\phi(S))$ and, if $\phi(S)$ is $\sigma$-invariant, so is $\phi(\mathbb{C}(S))$. Using that $\sigma(i)=-i$ and that $\sigma$ fixes $\phi(\mathbb{R}(T))$ point-wise, we have, in particular, that $\phi(\mathbb{C}(S))$ is $\sigma$-invariant for every subset $S$ of $\mathbb{C}(T, \sigma)$ which is contained in $\mathbb{R}(T) \cup i \mathbb{R}(T)$.
\end{example}

The next result, which is a partial L\"uroth's theorem over $\mathbb{C}(T, \sigma)$, is the aim of this section:

\begin{theorem} \label{thm:main_2}
Let $\mathbb{C} \subseteq L \subseteq \mathbb{C}(T, \sigma)$ be an intermediate division ring such that $\phi(L)$ is $\sigma$-invariant. Then there is $v \in \mathbb{C}(T, \sigma)$ such that $L = \mathbb{C}(v)$.
\end{theorem}

The key idea for proving Theorem \ref{thm:main_2} is to apply Theorem \ref{thm:main_1} to a suitable extension of $L$ inside $\mathbb{H}(X)$. We will first need the following lemma:

\begin{lemma}\label{swap} 
{\rm{(1)}} For every $f \in \mathbb{C}(T, \sigma)$, we have $j \phi(f) = \sigma(\phi(f)) j$. 

\vspace{0.5mm}

\noindent
{\rm{(2)}} We have $j \not \in \phi(\mathbb{C}(T, \sigma))$. In particular, the $\phi(L)$-linear space $\phi(L) + \phi(L)j$ has dimension 2 for every division ring $L$ contained in $\mathbb{C}(T, \sigma)$.

\vspace{0.5mm}

\noindent
{\rm{(3)}} Let $L$ be a division ring contained in $\mathbb{C}(T, \sigma)$ such that $\phi(L)$ is $\sigma$-invariant. Then $\phi(L) + \phi(L)j$ is a division ring.

\vspace{0.5mm}

\noindent
{\rm{(4)}} For $i \in \{1,2\}$, fix a division ring $L_i$ contained in $\mathbb{C}(T, \sigma)$. Assume $\phi(L_1) \subseteq \phi(L_2)$ and $\phi(L_2) + \phi(L_2)j \subseteq \phi(L_1) + \phi(L_1)j$. Then $L_1=L_2$.
\end{lemma}

\begin{proof}[Proof of Lemma \ref{swap}]
(1) By additivity and the definition of $\phi$, it suffices to consider the case where $f$ is a monomial $aT^n$ ($n \in \mathbb{N}$, $a \in \mathbb{C}$). As $j\phi(f)= jaj^n X^n$ and $\sigma(\phi(f))j = \sigma(a j^n)X^n j = \sigma(a) j^n X^n j= \sigma(a) j^{n+1} X^n$, it suffices to check that $ja=\sigma(a)j$, which was already mentioned in the proof of Lemma \ref{embedding}. 
\vspace{0.5mm}

\noindent
(2) If $j \in \phi(\mathbb{C}(T, \sigma))$, there are non-negative integers $n, m$ and elements $a_0, \dots, a_n, b_0, \dots, b_m$ of $\mathbb{C}$ such that $b_0 + b_1 jX + \cdots + b_m j^m X^m \not=0$ and such that
$$a_0 + a_1 jX + \cdots + a_nj^n X^n = j(b_0 + b_1 jX + \cdots + b_m j^m X^m).$$
Therefore, $a_\ell = jb_\ell$ for $0 \leq \ell \leq n=m$, i.e., every coefficient $b_\ell$ equals 0, which cannot hold.

\vspace{0.5mm}

\noindent
(3) Fix $a, b,c,  d \in L$. By (1), we have
$$(\phi(a)+\phi(b)j)(\phi(c)+\phi(d)j) = (\phi(a) \phi(c) - \phi(b) \sigma(\phi(d))) + (\phi(a) \phi(d) + \phi(b) \sigma(\phi(c)))j.$$
As we assumed that $\phi(L)$ is $\sigma$-invariant, $\phi(a) \phi(c) - \phi(b) \sigma(\phi(d))$ and $\phi(a) \phi(d) + \phi(b) \sigma(\phi(c))$ are in $\phi(L)$. Hence, $\phi(L) + \phi(L)j$ is a ring, which has no zero divisors. As $\phi(L) + \phi(L)j$ is also a finite dimensional $\phi(L)$-linear space, it is in fact a division ring (see, e.g., \cite[Proposition 3.1.2]{Coh95}).

\vspace{0.5mm}

\noindent
(4) By (2), we have
$$\begin{array}{lllll}
2 &=& {\rm{dim}}_{\phi(L_1)} \, \phi(L_1) + \phi(L_1)j &\geq& {\rm{dim}}_{\phi(L_1)} \, \phi(L_2) + \phi(L_2)j \\
&&&= &[\phi(L_2):\phi(L_1)] \cdot  {\rm{dim}}_{\phi(L_2)} \,  \phi(L_2) + \phi(L_2)j  \\
&&&  = &2 [\phi(L_2):\phi(L_1)].
\end{array}$$
Hence, $\phi(L_1) = \phi(L_2)$, i.e., $L_1=L_2$.
\end{proof}

\begin{proof}[Proof of Theorem \ref{thm:main_2}]
First, let us introduce the following map:
$$\tau : \left \{ \begin{array} {ccc}
\mathbb{H} & \longrightarrow & \mathbb{H} \\
a +bi + cj +dk & \longmapsto & a+bi-cj-dk
\end{array} \right.. $$
Then $\tau$ is an automorphism of $H$ of order 2. We may extend $\tau$ to an automorphism of $\mathbb{H}[X]$ of order 2, by setting
$$\tau(a_0 + a_1 X + \cdots  + a_n X^n) = \tau(a_0) - \tau(a_1) X + \cdots + (-1)^n \tau(a_n) X^n \quad (n \geq 0, a_0, \dots, a_n \in \mathbb{H}).$$
As $\mathbb{H}[X]$ is a right Ore domain, we may extend $\tau$ to an automorphism of $\mathbb{H}(X)$ of order 2, by setting
$$\tau (PQ^{-1}) = \tau(P) \tau(Q)^{-1} \quad (P \in \mathbb{H}[X], Q \in \mathbb{H}[X] \setminus \{0\}).$$

Now, consider the $\phi(L)$-linear space $\phi(L) + \phi(L)j$. Since $\tau$ fixes $\phi(\mathbb{C}(T, \sigma))$ point-wise and $\tau(j) = -j$, we have $\tau(\phi(a) + \phi(b)j) = \phi(a) + \phi(-b) j$ for $a,b \in L$. In particular,
\begin{equation} \label{eq:inc}
\tau(\phi(L) + \phi(L)j) = \phi(L) + \phi(L)j \quad {\rm{and}} \quad \phi(L) = \{u \in \phi(L) + \phi(L)j : \tau(u)=u\}.
\end{equation} 
Moreover, as $\phi(L)$ is $\sigma$-invariant, Lemma \ref{swap}(3) yields that $\phi(L) + \phi(L)j$ is a division ring, which is contained in $\mathbb{H}(X)$ and which contains $\phi(\mathbb{C})= \mathbb{C}$ and $j$, i.e., $\mathbb{H}$. By Theorem \ref{thm:main_1}, we then have 
\begin{equation} \label{eq:quat}
\phi(L) + \phi(L)j = \mathbb{H}(f)
\end{equation} 
for some $f \in \mathbb{R}(X)$. Let us fix $g, h \in \mathbb{R}(X^2)$ such that
\begin{equation} \label{decomposition}
f = g + h X.
\end{equation} 
Since $\tau(X) = -X$ and $\tau$ fixes $\mathbb{R}$ point-wise, we have $\tau(g)=g$ and $\tau(h)=h$. Hence, 
\begin{equation} \label{dec2}
\tau(f) = g-hX 
\end{equation}
and, by \eqref{eq:inc} and \eqref{eq:quat}, we get that $\tau(f) \in \phi(L) + \phi(L)j$. As $g= (f+ \tau(f))/2$ by \eqref{decomposition} and \eqref{dec2}, we get that $g \in \phi(L) + \phi(L)j$ and, as $g$ is invariant under $\tau$, we may apply \eqref{eq:inc} to get
\begin{equation} \label{eq:g}
g \in \phi(L).
\end{equation}
In particular, $jhX = j(f-g) \in \phi(L) + \phi(L)j$. Since $\tau(jhX)$ and $jhX$ coincide, \eqref{eq:inc} actually gives
\begin{equation} \label{eq:jhX}
jhX \in \phi(L).
\end{equation}

Next, as $f \in \mathbb{R}(X)$, $f$ is central in $\mathbb{H}(f)$. As $\tau(\mathbb{H}(f))=\mathbb{H}(f)$ by \eqref{eq:inc} and \eqref{eq:quat}, we get that $\tau(f)$ is central in $\mathbb{H}(f)$, i.e., $\tau(f) \in \mathbb{R}(f)$ (see Lemma \ref{lemma:dl}). As $\tau$ fixes $\mathbb{R}$ point-wise, the restriction of $\tau$ to $\mathbb{R}(f)$ is an $\mathbb{R}$-automorphism of order $\leq 2$. Assume $f$ is algebraic over $\mathbb{R}$. Then $\phi(L) + \phi(L)j = \mathbb{H}$ by \eqref{eq:quat}. As $\mathbb{H}=\phi(\mathbb{C}) + \phi(\mathbb{C})j$ and as $\phi(\mathbb{C}) \subseteq \phi(L)$, we may apply Lemma \ref{swap}(4) to get that $L= \mathbb{C}$.

Therefore, assume $f$ is transcendental over $\mathbb{R}$. Hence, $\tau|_{\mathbb{R}(f)}$ is a M{\"o}bius transformation of order at most 2. First, assume $\tau(f)=f$, i.e., $g-hX = g+hX$ by \eqref{decomposition} and \eqref{dec2}. Hence, $h=0$ and 
\begin{equation} \label{equality}
\phi(L) + \phi(L)j = \mathbb{H}(g)
\end{equation}
by \eqref{eq:quat}. Moreover, as $g = \phi(g(-T^2)) \in \phi(\mathbb{R}(T))$, we get that $\mathbb{C}(g) = \phi(\mathbb{C}(g(-T^2)))$ is $\sigma$-invariant (see Example \ref{example}). We may then apply Lemma \ref{swap}(3) to get that
$\mathbb{C}(g) + \mathbb{C}(g)j$ is a division ring. But $\mathbb{C}(g) + \mathbb{C}(g)j$ contains $\mathbb{C}$, $g$ and $j$, i.e., contains $\phi(L) + \phi(L)j$ by \eqref{equality}. As \eqref{eq:g} yields $\mathbb{C}(g) \subseteq \phi(L)$, we may then apply Lemma \ref{swap}(4) to get $L = \mathbb{C}(g(-T^2))$.

From now on, assume $\tau(f) \not=f$. Set 
\begin{equation} \label{eq:f}
\tau(f) = \frac{af+b}{cf+d}
\end{equation}
with $a,b,c,d \in \mathbb{R}$ satisfying $ad-bc \neq 0$. We then have
\begin{equation} \label{eq:f2}
f = \tau^2(f) = \frac{(a^2+bc)f+b(a+d)}{c(a+d)f+cb+d^2}.
\end{equation}
First, assume $c = 0$, in which case \eqref{eq:f} and \eqref{eq:f2} reduce to
$$\tau(f) = \frac{af+b}{d} \quad {\rm{and}} \quad f = \frac{a^2f + b(a+d)}{d^2},$$
respectively. In particular, from the second equality, we get that $f(1-a^2/d^2)$ is a real number, which is possible only if $a^2=d^2$. If $a=d$, then the second equality yields further $b=0$. Hence, $\tau(f)=f$ by the first equality, which cannot hold. Therefore, $a=-d$ and the first equality yields $\tau(f) = -f-ba^{-1}$. Then, by \eqref{decomposition} and \eqref{dec2}, it follows that $g=-b(2a)^{-1} \in \mathbb{R}$. Thus 
\begin{equation} \label{dec3}
\phi(L) + \phi(L)j = \mathbb{H}(hX)
\end{equation} 
by \eqref{eq:quat} and \eqref{decomposition}. Moreover, as $jhX = \phi(Th(-T^2)) \in \phi(\mathbb{R}(T))$, we get that $\mathbb{C}(jhX) = \phi(\mathbb{C}(Th(-T^2)))$ is $\sigma$-invariant (see Example \ref{example}). Therefore, Lemma \ref{swap}(3) yields that $\mathbb{C}(jhX) + \mathbb{C}(jhX)j$ is a division ring. But $\mathbb{C}(jhX) + \mathbb{C}(jhX)j$ contains $\mathbb{C}$, $jhX$ and $j$, i.e., contains $\phi(L) + \phi(L)j$ by \eqref{dec3}. As \eqref{eq:jhX} yields $\mathbb{C}(jhX) \subseteq \phi(L)$, we get $L = \mathbb{C}(Th(-T^2))$ from Lemma \ref{swap}(4).

Finally, assume $c \neq 0$. Then divide $a$, $b$ and $d$ by $c$ to assume that $c = 1$. We then have 
$$\tau(f) = \frac{af+b}{f+d} \quad {\rm{and}} \quad f = \frac{(a^2 + b)f  + b(a+d)}{(a+d)f + b+ d^2},$$
by \eqref{eq:f} and \eqref{eq:f2}. By the second equality, we get $d=-a$ and the first equality then yields
$$\tau(f) = \frac{af+b}{f-a}.$$
In particular, 
$$\tau(f-a) = \frac{af+b-a(f-a)}{f-a} = \frac{b+a^2}{f-a}.$$
As $a \in \mathbb{R}$, we have $\mathbb{R}(f) = \mathbb{R}(f-a)$ and, hence, $\mathbb{H}(f) = \mathbb{H}(f-a)$. Therefore, we may replace $f$ with $f-a$ to assume $\tau(f) = 1/(\alpha f)$, where $\alpha = 1/(b+ a^2)$. Set $f=P(X)/Q(X)$, where $P$ and $Q$ are coprime polynomials with real coefficients. Then
$$\frac{P(-X)}{Q(-X)} = \tau(f) = \frac{1}{\alpha f} = \frac{1}{\alpha} \cdot \frac{Q(X)}{P(X)},$$
i.e., $\alpha P(-X) P(X) = Q(-X) Q(X)$. In particular, $\alpha P(0)^2 = Q(0)^2$. If $P(0)=0$, then we also have $Q(0)=0$, which cannot hold since $P$ and $Q$ are coprime. Therefore, $\alpha = Q(0)^2 / P(0)^2 >0$.

Given a square root $\sqrt{\alpha}$ of $\alpha$ in $\mathbb{R}$, set
$$v = \frac{1-\sqrt{\alpha}f}{1+\sqrt{\alpha}f} \cdot j \in \mathbb{H}(f) = \phi(L) + \phi(L)j.$$ 
Then
$$\tau(v) = \frac{(1-\sqrt{\alpha}(\alpha f)^{-1})}{1+\sqrt{\alpha} (\alpha f)^{-1}} \cdot (-j) = \frac{(\alpha f - \sqrt{\alpha})}{\alpha f + \sqrt{\alpha}} \cdot (-j) = \frac{\big(1-\sqrt{\alpha} f\big)}{1+\sqrt{\alpha}f } \cdot j = v,$$
thus yielding $\mathbb{C}(v) \subseteq \phi(L)$ by \eqref{eq:inc}. If $\phi(L) + \phi(L)j \subseteq \mathbb{C}(v) + \mathbb{C}(v) j$, then Lemma \ref{swap}(4) yields $L= \mathbb{C}(w)$, where $w$ is the unique element of $\mathbb{C}(T, \sigma)$ fulfilling $\phi(w)=v$. To get the desired inclusion, note first that, as $f \in \mathbb{R}(X)$ and as $\sqrt{\alpha} \in \mathbb{R}$, we have $\sigma(v)=v$. Therefore, $\mathbb{C}(v)$ is $\sigma$-invariant and, by Lemma \ref{swap}(3), we get that $\mathbb{C}(v) + \mathbb{C}(v)j$ is a division ring. As $\phi(L) + \phi(L)j =\mathbb{H}(f)$ by \eqref{eq:quat} and as $\mathbb{H} \subseteq \mathbb{C}(v) + \mathbb{C}(v)j$, it then suffices to show that $f \in \mathbb{C}(v)  + \mathbb{C}(v) j$. But, since
$$f \mapsto \frac{1-\sqrt{\alpha}f}{1+\sqrt{\alpha}f}$$ 
is a M{\"o}bius transformation of the rational function field $\mathbb{R}(f)$, there are $a_1, a_2, a_3, a_4 \in \mathbb{R}$ such that
$$f = \bigg(a_1 \frac{1-\sqrt{\alpha}f}{1+\sqrt{\alpha}f} + a_2 \bigg) \bigg(a_3 \frac{1-\sqrt{\alpha}f}{1+\sqrt{\alpha}f} + a_4 \bigg)^{-1} = (a_2 -a_1 v j) (a_4 - a_3 vj)^{-1}$$
and, since $\mathbb{C}(v) + \mathbb{C}(v)j$ is a division ring, we have $(a_2 -a_1 v j) (a_4 - a_3 vj)^{-1} \in \mathbb{C}(v) + \mathbb{C}(v)j$.
\end{proof}

The following proposition shows that $\sigma$-invariance is not necessary in general for intermediate division rings $\mathbb{C} \subseteq L \subseteq \mathbb{C}(T, \sigma)$ to be of the form $\mathbb{C}(v)$ with $v \in \mathbb{C}(T, \sigma)$:

\begin{proposition} \label{converse}
Set $v = T +i T^3 \in \mathbb{C}(T, \sigma)$. Then $\phi(\mathbb{C}(v))$ is not $\sigma$-invariant.
\end{proposition} 

\begin{proof} 
First, note that $\phi(v) = jX + ij^3 X^3 = jX - ij X^3 = jX - kX^3$. 

Now, assume $\phi(\mathbb{C}(v))$ is $\sigma$-invariant. Since $\phi(\mathbb{C}(v)) = \mathbb{C}(\phi(v)) = \mathbb{C}(jX - kX^3)$ (see Example \ref{example} for the first equality), we have $\sigma(jX - kX^3) = jX + kX^3 \in \mathbb{C}(jX - kX^3)$. Therefore, $\mathbb{C}(jX - kX^3)$ contains $jX = \phi(T)$. Consequently, $\mathbb{C}(v)$ contains $T$, i.e., 
\begin{equation} \label{eq:ct}
\mathbb{C}(v)=\mathbb{C}(T, \sigma).
\end{equation} 

Next, let $R$ denote the intersection of all subrings of $\mathbb{C}[T, \sigma]$ which contain both $\mathbb{C}$ and $v$. Since $vi=-iv$, we have
$$R = \{a_0 + a_1v + \cdots + a_n v^n : n \geq 0, a_0, \dots, a_n \in \mathbb{C}\}.$$
Moreover, since $v$ has positive degree, the $v^n$'s ($n \geq 0$) are linearly independent over $\mathbb{C}$. Therefore, $R$ is the polynomial ring $\mathbb{C}[v, \sigma]$. In particular, $\mathbb{C}(v)$, which is the intersection of all division rings contained in $\mathbb{C}(T, \sigma)$ and containing both $\mathbb{C}$ and $v$, equals $\mathbb{C}(v, \sigma)$. Hence, by, e.g., \cite[lemme 2.3]{BDL22}, the center of $\mathbb{C}(v)$ equals $\mathbb{R}(v^2) = \mathbb{R}(T^2 + T^6)$. By \eqref{eq:ct}, we then obtain
\begin{equation} \label{center}
\mathbb{R}(T^2) = \mathbb{R}(T^2 + T^6).
\end{equation}

Finally, note that $T^2$ is a root of $X^3 + X - (T^2 + T^6) \in \mathbb{R}(T^2 + T^6)[X]$ and that $T^2 + T^6$ is transcendental over $\mathbb{R}$. Considering a transcendental $Y$, it is easily checked that $X^3 + X - Y$ has no root in $\mathbb{R}(Y)$, i.e., that $X^3 + X - Y$ is irreducible over $\mathbb{R}(Y)$. Consequently, $\mathbb{R}(T^2) = \mathbb{R}(T^2+T^6)(T^2)$ is a degree 3 extension of $\mathbb{R}(T^2+T^6)$, which contradicts \eqref{center}.
\end{proof}

To summarize, we have the following combination of Theorem \ref{thm:main_2} and Proposition \ref{converse}:

\begin{corollary} 
Let $\mathbb{C} \subseteq L \subseteq \mathbb{C}(T,\sigma)$ be an intermediate division ring. For $L$ to be of the form $\mathbb{C}(v)$ with $v \in \mathbb{C}(T, \sigma)$, it is sufficient, but not necessary in general, that $\phi(L)$ is $\sigma$-invariant. 
\end{corollary}

Finding a precise, unconditional version of L\"uroth's Theorem for division rings of the form $D(T,\sigma)$ (where $D$ is a division algebra) remains an open question. More generally, one may ask for a precise version of Igusa's Theorem for skew function fields of higher dimension. The following example demonstrates another type of obstruction for such a theorem.

\begin{example}\label{referee_example} Let $R$ be the first Weyl algebra over the complex numbers $\mathbb{C}$, with generators $X,Y$. That is, $R$ is the quotient of the free $\mathbb{C}$-algebra in $X,Y$ by the ideal $\langle XY-YX -1 \rangle$. Let $K$ be the first Weyl skew field, that is, the quotient skew field of $R$. Then $K$ has transcendence degree $2$ over $\C$, in the sense of Gelfand-Kirillov (see \cite{GK66}). The skew field $K$ contains a (commutative) subfield $L$, generated over $\C$ by elements $a,b$ satisfying $a^2-b^3=1$, by a theorem of Dixmier \cite[Proposition 5.5]{Dix68}, and such a subfield $L$ is not generated over $\C$ by a single generator. \end{example}

\appendix

\section{Proof of Lemma \ref{lemma:finite}} \label{sec:finite}

Firstly, set $Z(H) \cdot Z(L) = \{x_1y_1 + \cdots + x_n y_n : n \geq 1, x_1, \dots, x_n \in Z(H), y_1, \dots, y_n \in Z(L)\}.$ Then $Z(H) \cdot Z(L)$ contains $Z(H)$ and $Z(L)$, is contained in $L$, is a commutative ring, and is a $Z(L)$-subspace of $L$. Moreover, $Z(H) \cdot Z(L)$ is an integral domain (since $L$ is a division ring) and ${\rm{dim}}_{Z(L)} \, Z(H) \cdot Z(L) \leq {\rm{dim}}_{Z(L)} \, L < \infty$. Hence, $Z(H) \cdot Z(L)$ is a field.

Secondly, set $H \cdot Z(L) = \{x_1y_1 + \cdots + x_n y_n : n \geq 1, x_1, \dots, x_n \in H, y_1, \dots, y_n \in Z(L)\}.$ Then $H \cdot Z(L)$ contains $H$ and $Z(L)$, is contained in $L$, is a ring, and is a $Z(L)$-subspace of $L$ with ${\rm{dim}}_{Z(L)} \, H \cdot Z(L) < \infty$. In fact, $H \cdot Z(L)$ is a $(Z(H) \cdot Z(L))$-subspace of $L$ and we have ${\rm{dim}}_{Z(H) \cdot Z(L)} \,  H \cdot Z(L) < \infty$. Moreover, $Z(H) \cdot Z(L) \subseteq Z(H \cdot Z(L))$.

Thirdly, let $\{e_i\}_{i \in I}$ be a $Z(H)$-basis of $H$. We claim that $\{e_i\}_{i \in I}$ is a $(Z(H) \cdot Z(L))$-basis of $H \cdot Z(L)$. Since ${\rm{dim}}_{Z(H) \cdot Z(L)} \,  H \cdot Z(L) < \infty$, we get that $I$ is finite, as needed for the lemma.

To show the claim, we first show that the $e_i$'s span $H \cdot Z(L)$ over $Z(H) \cdot Z(L)$. To that end, fix $n \geq 1$, $x_1, \dots, x_n \in H$ and $y_1, \dots, y_n \in Z(L)$. Then there is a finite subset $J$ of $I$ such that, for $k \in \{1, \dots, n\}$, there are elements $\lambda_{k,j}$ ($j \in J$) of $Z(H)$ verifying $x_k = \sum_{j \in J} \lambda_{k,j}e_j.$ We then have
$$x_1y_1 + \cdots + x_n y_n = \bigg(\sum_{j \in J} \lambda_{1,j}e_j \bigg) y_1 + \cdots + \bigg(\sum_{j \in J} \lambda_{n,j}e_j \bigg) y_n = \sum_{j \in J} (\lambda_{1,j} y_1 + \cdots + \lambda_{n,j}y_n) e_j.$$
Finally, let $\{i_1, \dots, i_n\} \subseteq I$ and $\lambda_1, \dots, \lambda_n \in Z(H) \cdot Z(L)$ be such that $\lambda_1 e_{i_1} + \cdots + \lambda_n e_{i_n} =0$, i.e., 
\begin{equation} \label{eq:lin}
e_{i_1} \lambda_1 + \cdots + e_{i_n} \lambda_n =0
\end{equation}
(as $Z(H) \cdot Z(L) \subseteq Z(H \cdot Z(L))$). As $(x,y) \in H \times (Z(H) \cdot Z(L)) \mapsto xy \in H \cdot Z(L)$ is $Z(H)$-bilinear, there is a unique $Z(H)$-linear map $\psi : H \otimes_{Z(H)} (Z(H) \cdot Z(L)) \rightarrow H \cdot Z(L)$ which fulfills $\psi(x \otimes y) = xy$ for every $(x,y) \in H \times (Z(H) \cdot Z(L))$. Moreover, as $Z(H) \cdot Z(L) \subseteq Z(H \cdot Z(L))$, the map $\psi$ is a morphism of $Z(H)$-algebras and, by, e.g., \cite[Proposition 2.36]{Kna07}, it is injective. Therefore, by \eqref{eq:lin}, we have
\begin{equation} \label{eq:lin2}
e_{i_1} \otimes \lambda_1 + \cdots + e_{i_n} \otimes \lambda_n =0.
\end{equation}
Now, fix $Z(H)$-linear maps $f_1, \dots, f_n : Z(H) \cdot Z(L) \rightarrow Z(H)$ and, for $j \in \{1, \dots, n\}$, set 
$$e_{i_j}^* : \left \{ \begin{array} {ccc}
H & \rightarrow & Z(H) \\
e_i & \mapsto & \delta_{i, i_j}
\end{array} \right.,$$
where $\delta_{i, i_j}$ denotes the Kronecker symbol. Since 
$$F : \left \{ \begin{array} {ccc}
H \times (Z(H) \cdot Z(L)) & \rightarrow & Z(H) \\
(x,y) & \mapsto &  e_{i_1}^*(x) f_1(y) + \cdots + e_{i_n}^*(x) f_n(y)
\end{array} \right.$$
is $Z(H)$-bilinear, there is a unique $Z(H)$-linear map $\widetilde{F} : H \otimes_{Z(H)} (Z(H) \cdot Z(L)) \rightarrow Z(H)$ which fulfills $\widetilde{F}(x \otimes y) = F(x,y) = e_{i_1}^*(x) f_1(y) + \cdots + e_{i_n}^*(x) f_n(y)$ for every $(x,y) \in H \times (Z(H) \cdot Z(L))$. By \eqref{eq:lin2}, we then have
$$0 = \widetilde{F}(e_{i_1} \otimes \lambda_1 + \cdots + e_{i_n} \otimes \lambda_n) = f_1(\lambda_1) + \cdots + f_n(\lambda_n).$$
In particular, fixing $j \in \{1, \dots, n\}$ and setting $f_1 = \cdots = f_{j-1} = f_{j+1} = \cdots = f_n=0$, we get $f_j(\lambda_j)=0$. Fix a $Z(H)$-basis $\{\epsilon_i\}_{i \in I'}$ of $Z(H) \cdot Z(L)$ and, for $i \in I'$, set 
$$\epsilon_i^* : \left \{ \begin{array} {ccc}
Z(H) \cdot Z(L) & \rightarrow & Z(H) \\
\epsilon_{i'} & \mapsto & \delta_{i, i'}
\end{array} \right..$$
As $f_j$ was arbitrary, we get $\epsilon_i^*(\lambda_j)=0$ for every $i \in I'$, i.e., $\lambda_j=0$. This concludes the proof.

\bibliography{Biblio2}
\bibliographystyle{alpha}

\end{document}